\documentclass[12pt]{amsart} 
\usepackage{mathrsfs}
\usepackage{amssymb,amsmath,amsthm,color}
\usepackage{graphicx}
\usepackage{hyperref}
\usepackage{url} 
\usepackage{setspace}
\usepackage{mathtools}
\usepackage[inline]{enumitem}  
\usepackage[margin=1in]{geometry}
\usepackage{comment}
\usepackage{cite}
\numberwithin{equation}{section}
\newtheorem{theorem}{Theorem}[section]
\newtheorem*{A1}{{\bf Theorem A}}
\newtheorem*{B1}{{\bf Theorem B}}

\newtheorem{Remark}{Remark}[section]
\newtheorem{proposition}{Proposition}[section]
\newtheorem{Lemma}{Lemma}[section]

\newcommand{\R}{\mathbb R}

\newcommand{\N}{\mathbb N} 

\newcommand{\m}{\mathcal{M}}

\begin{document}
\baselineskip16pt
\title[The blow-up solutions for fractional heat equation ]{The Blow-up solutions for fractional heat equations on torus and Euclidean space}
\author{Divyang G. Bhimani}
\address{Department of Mathematics, Indian Institute of Science Education and Research, Dr. Homi Bhabha Road, Pune 411008, India}
\email{divyang.bhimani@iiserpune.ac.in}

\subjclass[2010]{35K05, 35K55, 42B35 (primary), 35A01 (secondary)}
\keywords{fractional heat equations, blow-up solution, modulation spaces,  Fourier amalgam spaces}
\date{}
\maketitle
\begin{abstract}  
We  produce  a  finite time  blow-up solution for  nonlinear  fractional heat equation ($\partial_t u + (-\Delta)^{\beta/2}u=u^k$) in modulation and Fourier amalgam  spaces on the torus $\mathbb T^d$ and  the Euclidean space $\mathbb R^d.$
This complements  several  known  local  and  small data global  well-posedness results in modulation spaces on $\mathbb R^d.$ Our method of proof  rely on the  formal solution of the equation.  This method should be further applied to other non-linear evolution equations.
 \end{abstract}
\section{Introduction}
 We study  heat equation  associated to fractional Laplacian 
 $(-\Delta)^{\beta/2}$ of the form 
\begin{eqnarray}\label{fh}
\begin{cases}\partial_t u + (-\Delta)^{\beta/2}u=u^k,  \qquad (t,x)\in \mathbb R^{+}\times \mathcal{M}\\
u(0, x)= u_0(x), 
\end{cases}
\end{eqnarray} 
where  $\mathcal{M}=\mathbb T^d$ (torus) or $\mathbb R^d$.  We  denote   $\widehat{\mathcal{M}}$  the Pontryagin dual of $\mathcal{M},$ i.e.  
 \begin{equation*}
 \widehat{\mathcal{M}}=
 \begin{cases} \mathbb Z^d \quad if \ \mathcal{M}= \mathbb T^d\\
 \mathbb R^d \quad if \ \mathcal{M}=\mathbb R^d.
 \end{cases}
 \end{equation*}
The fractional Laplacian $(-\Delta)^{\beta/2}$  is defined by 
\begin{eqnarray*}
\mathcal{F}[(-\Delta)^{\beta/2}u] (\xi) =  c_{\beta}|\xi|^{\beta} \mathcal{F}u (\xi) \qquad (\xi \in \widehat{ \mathcal{M}})
\end{eqnarray*}
where $\mathcal{F}$ denotes the Fourier transform and $c_{\beta}$ is some constant.  The fractional heat equation   \eqref{fh} is significantly interesting in both physics and PDEs, since it is the Poisson equation if  $\beta =1$ and the classical heat equation if   $\beta =2.$ In the later case   \eqref{fh}  appears  as a one-dimensional model for the voraticity equation of incompressible and viscous fluid of three dimension.  

In the 1980s, Feichtinger~\cite{feichtinger1983modulation} introduced the \textbf{modulation spaces} $M^{p,q}_s(\m)$ 
using short-time Fourier transform (STFT)~ \footnote{{STFT is} also known as windowed Fourier transform and is closely related to Fourier--Wigner and Bargmann transform.  See~\cite[Lemma 3.1.1]{grochenig2013foundations}  and~\cite[Proposition 3.4.1]{grochenig2013foundations} .}. The STFT of a $f\in \mathcal{S}'(\m)$ (space of tempered distributions, see e.g.  \cite[Part II]{RuzhanksyBook}) with
respect to a window function $0\neq g \in {\mathcal S}(\m)$ (Schwartz space) is defined by 
\begin{equation*}\label{stft}
V_{g}f(x,y)= \int_{\m} f(t) \overline{T_xg(t)} e^{- 2\pi i y\cdot t}dt, \ (x, y) \in \m \times \widehat{\m}
\end{equation*}
whenever the integral exists. Here, $T_xg(t)=g(tx^{-1})$ is the translation operator on $\m$.   The modulation spaces $M^{p,q}_s(\mathcal{M})$ is defined as follows:
\[M^{p,q}_s(\mathcal{M})=\left\{f\in \mathcal{S}'(\m):\|f\|_{M^{p,q}_s}= \left\| \|V_gf(x,y)\|_{L^p(\m)} \langle y \rangle^s \right\|_{L^q(\widehat{\m})}<\infty \right\}.\]
See also Remark \ref{evd}.  If $s=0,$ we write $M^{p,q}_0(\m)=M^{p,q}(\m).$  
It is known that
\begin{eqnarray*}
M^{p,q}_s(\m)=
\begin{cases} H^{s}(\m) \ \text{(Sobolev space)} \quad if \  $p=q=2$\\
\mathcal{F}L^q(\m) \ \text{(Fourier-Lebesgue space)} \quad if  \ \m=\mathbb T^d.
\end{cases}
\end{eqnarray*} 
In last two decades modulation spaces 
have turned out to be very fruitful in the study of various non-linear PDEs on $\mathbb R^d$,  see \cite{baoxiang2006isometric, wang2007global,wang2011harmonic, benyi2009local, bhimani2016functions, BhimaniHartree-Fock,BhimaniNorm, BhimaniAdM}.   
Iwabuchi  \cite[Theorems 1.9 and 1.13]{iwabuchi2010navier} proved local and  global well-posedness   of \eqref{fh}  with $\beta =2$  for small data in some weighted modulation spaces $M^{p,q}_s(\mathbb R^d)$.  Later  Chen-Deng-Ding-Fan \cite[Theorems 3.1, 1.5 and 1.6]{chen2012estimates} have  obtained some space-time estimates for heat  semigroup  $e^{-t (-\Delta)^{\beta/2}}$ in modulation spaces.  Further,  they  \cite[Theorems  1.5 and 1.6]{chen2012estimates}  proved  local and  global well-posedness (for small data)  of \eqref{fh}  with  any $\beta>0$ in some weighted modulation spaces, see also \cite[Theorem 3]{ru2014multilinear}. In \cite{huang2016critical, zheng1}  authors have found some critical  exponent in modulation spaces and provide some  local well-posedness and ill-posedness   for  \eqref{fh} in  weighted  modulation spaces.   Perhaps the best known  local well-posedness for \eqref{fh} read as follows: 
\begin{A1}[\cite{chen2012estimates}]\label{a}  Assume that $u_0 \in M^{p,1}(\R^d) \ (1\leq p \leq \infty).$  Then there exists $T_1>0$ such that  \eqref{fh} has a unique solution $u \in C([0, T_1), M^{p,1}(\R^d)).$
\end{A1}
 We note that though there is extensive literature on classical heat equation on $\mathbb R^d,$  only a few authors have studied \eqref{fh} on the torus $\mathbb T^d$ (in spite  of  periodic data is very important in  the analysis and in applications).   See e.g.  \cite{hiroheat}.  On the other hand, there has been good well-posedness theory developed for non-linear Schr\"odinger equations on torus $\mathbb T^d$ but not for \eqref{fh}.   We  are thus also interested to study time behaviour solution of \eqref{fh} for periodic data.  
 
 Recently in \cite{JH},   Forlano and Oh  have introduced \textbf{Fourier amalgam spaces} $\widehat{w}^{p,q}_s(\mathcal{M}) \ (1\leq p, q \leq \infty):$
$$\widehat{w}^{p,q}(\mathcal{M})=\left\{ f\in \mathcal{S}'(\mathcal{M}): \|f\|_{\widehat{w}^{p,q}}= \left\|\|\chi_{n+(-\frac{1}{2},\frac{1}{2}]^d} (\xi)\mathcal{F}f(\xi)\|_{L_{\xi}^p(\widehat{\mathcal{M}})}\right\|_{\ell_n^q(\mathbb Z^d)}< \infty \right\}.
$$
Their  motivation was to study  well-posedness of 1D cubic NLS in these spaces, \cite{forlano2020deterministic}.  See also Remark \ref{chr}.
 It is interesting to note that $\widehat{w}^{q,q}(\mathcal{M})=\mathcal{F}L^q(\mathcal{M}),$  $\widehat{w}^{2,q}(\mathcal{M})= M^{2,q}(\m)$ and $\widehat{w}^{p,1}(\mathcal{M}) \ (1\leq p \leq \infty)$
 is an algebra   under point-wise multiplication.   We also note that free fractional heat propagator  $e^{-t(-\Delta)^{\beta/2}}$ is uniformly bounded on $\widehat{w}^{p,q}(\mathcal{M}).$  Thus,  by the  standard fixed point argument we have the following local well-posedness theorem.
\begin{B1}\label{b}  Assume that $u_0 \in  \widehat{w}^{p,1}(\mathcal{M})\ (1\leq p \leq \infty).$  Then there exists $T_1>0$ such that  \eqref{fh} has a unique solution $u \in C([0, T_1),  \widehat{w}^{p,1}(\mathcal{M})).$
\end{B1}
For simplicity of presentation,  we let 
\[X(\mathcal{M})=M^{p,1}(\mathcal{M}) \ (1\leq p \leq 2) \quad \text{or} \quad \widehat{w}^{p,1}(\mathcal{M}) \ (1\leq p \leq \infty). \]

It is  now natural to ask whether   the local solutions established in Theorems A and B can be extended to  a global solution in time?  The purpose of the present paper is to answer this question.

In fact, for  $u_0 \in  X(\m),$  in view of Theorem A (see Proposition \ref{lw} below),  \eqref{fh} has  a unique solution  $u\in C([0, T^*), X(\m)),$ where $T^*=T^*(\|u_0\|_{X})$ denotes the  maximal  existence  of time  of solution. Moreover, we have  either $$T^*= \infty$$ or 
\[ T^{\ast}<\infty \quad \text{and} \quad \lim_{t\to T*} \sup \|u(t)\|_{X}=\infty.\]
In the previous case, we say that  the solution is global, while in the latter case we say that the solution \textbf{blows up in  the $M^{p,1}$ norm in  finite time }and  $T^{\ast}$ is called the blow-up time. We now state our  main theorem.
\begin{theorem}\label{dip}
Suppose that  $k\in \mathbb N, 0<r, \gamma, \beta  < \infty$ and $r^{d} v_d \geq 2^d,$   $v_d$ is the volume of unit ball $\{x\in \m:|x|\leq 1\}.$ Let $\widehat{u_0} \geq 0,  \widehat{u_0}\geq \gamma \chi_{B_0(r)},  u_0\in X (\m) \ (1\leq p \leq 2)$  and $\gamma^{k-1}\geq 4 r^{\beta} (k-1)e.$ Then \eqref{fh}  has  a unique blow-up solution in $X(\m),$ that is,  there exists a unique solution  $u(t)$  of \eqref{fh}  defined on $[0, T^*)$ such   that 
\[ T^*< \infty \quad \text{and} \quad  \limsup_{t\to T*} \|u(t)\|_{X}= \infty. \]
\end{theorem}
 We have initiated the study of blow-up analysis for  \eqref{fh} on torus $\mathbb T^d$ and  Theorem \ref{dip} thus is first finite time blow-up result in $\mathcal{F}L^1(\mathbb T^d)$ as far as we are aware. Theorem \ref{dip}  reveals that the  local solution established in  Theorem A  cannot be extended to a global solution in time.  Our method of proof uses  a formal solution of \eqref{fh} in terms of power series expansion. This formal solution satisfies the corresponding  Duhamel's form of \eqref{fh} (see Lemma \ref{me1}). Finally, by performing  the analysis on the Fourier transform side for component  of power series expansion, we establish crucial lower bound  for each component. This leads us to  blow-up solution in finite time (Lemma \ref{fl}).    We note that our method of proof is inspired by the recent work of Ru and Chen in \cite{ru2015blow},   where they established finite time-blow up for  a classical heat equation,  i.e.  \eqref{fh}
with $\beta=2,$ in  $\mathcal{F}L^1(\R^d).$

\begin{Remark}  (1) The function  
 $u_0(x)=  C e^{-2\pi |x|^2}$ for  $C> \sqrt[k-1]{4er^{\beta} (k-1)} e^{2\pi r^{\beta}}, r^d v_d  \geq 1$ satisfies the hypothesis of Theorem \ref{dip}.
(2) The function   $u_0(x)= C \frac{\sin x}{x},  x\in \R, C= \frac{1}{\pi} \sqrt[k-1]{4e(k-1)}$ satisfies the hypothesis of Theorem \ref{dip} for $p=2.$ (3) Taking Theorem A into account, the case $p>2$ in Theorem \ref{dip} remains interesting open question.  
\end{Remark}
\begin{Remark} The cubic NLS 
\begin{eqnarray*}
i\partial_{t}u+ \Delta u = |u|^2 u
\end{eqnarray*}
 is locally \cite{benyi2009local}  well-posed in $M^{p,1}(\mathbb R^d)$ but it is not yet clear whether  it is  globally well-posed or  there exist a blow-up solution, see for instance  open question raised by  Ruzhansky-Sugimoto-Wang  in 
\cite[p.280]{ruzhansky2012modulation}.  On the other hand,  Theorem \ref{dip} says that  we can produce   blow-up  solution of \eqref{fh} in $M^{p,1}(\R^d) \ (1\leq p\leq 2).$
\end{Remark}
We now briefly mention widely known literature.  Fujita \cite{fujita1966blowing} proved  that  for \eqref{fh} with $\beta=2$ and  $d(k-1)/2<1$  no non-negative global solution exists for     any non-trivial initial data  $u_0\in L^1(\R^d)$ (i.e.  every positive solution to this initial data problem blows up
in a finite time.) Later Fujita \cite{fujitasome} proved that, for $d(k-1)/2<1$,  global solutions do exist for initial data dominated by a sufficiently small  Gaussian.  For the critical exponent $d(k-1)/2=1,$ Hayakawa  \cite{hayakawa1973nonexistence}  proved
nonexistence of nonnegative and nontrivial global solutions in the case of  $d=1,2$, and Kobayashi-Sirao-Tanaka \cite{kobayashi1977growing}  proved it in general dimensions (i.e. finite time blow up case).  In \cite{ ru2015blow, brandolese2019blowup} authors have proved blow up in finite time  for  \eqref{fh} with $\beta=2$ in some scale invariant  Besov space and Fourier-Lebesgue  spaces. We note these  results deals with the classical heat equation  \eqref{fh}
with $\beta=2$ while in the present paper we could considered fractional heat equation \eqref{fh} with $\beta \neq 2$ as well.

\subsection{Further remarks}
\begin{Remark}B\'enyi and Okoudjou in \cite{benyi2009local} have established local well-posedness for nonlinear wave and Klein-Gordon (NLKG) equations in  $M^{p,1}(\mathbb R^d).$   Wang and Huzdik in \cite{wang2007global} have established small data global well-posedness for NLKG in  $M^{2,1}(\mathbb R^d)$.
Exploiting similar ideas of Section \ref{fsn}, we may establish  formal solution for the  wave and Klein-Gordon equations. The  only difference maybe at this stage is to replace fractional heat propagator  by  suitable free  Klein-Gordon and wave propagator. It is also further expected the method of proof of the present paper should be applicable to establish finite time blow up for these equations. 
\end{Remark}
\begin{Remark}\label{evd} Applying the frequency-uniform localization techniques,
one can get an equivalent definition of modulation spaces  \cite{wang2007global}  as follows.  
Let   $\rho \in \mathcal{S}(\mathbb R^d),$  $\rho: \mathbb R^d \to [0,1]$  be  a smooth function satisfying   $\rho(\xi)= 1 \  \text{if} \ \ |\xi|_{\infty}\leq \frac{1}{2} $ and $\rho(\xi)=
0 \  \text{if} \ \ |\xi|_{\infty}\geq  1.$ Let  $\rho_k$ be a translation of $\rho,$ that is, $ \rho_k(\xi)= \rho(\xi -k) \ (k \in \mathbb Z^d).$
Denote 
$\sigma_{k}(\xi)= \frac{\rho_{k}(\xi)}{\sum_{l\in\mathbb Z^{d}}\rho_{l}(\xi)}, \ (k \in \mathbb Z^d).$   
The frequency-uniform decomposition operators can be  defined by 
$$\square_k = \mathcal{F}^{-1} \sigma_k \mathcal{F}. $$
For $1\leq p, q \leq \infty, s\in \mathbb R,$  it is known \cite[Proposition 2.1]{wang2007global} that  
\begin{eqnarray*}\label{ed}
\|f\|_{M^{p,q}_s}\asymp  \left\|  \| \square_kf\|^q_{L^p_x(\m)} (1+ |k|)^s \right\|_{\ell_k^q}.
\end{eqnarray*}
The definition of the modulation space given above, is independent of the choice of 
the particular window function.  See \cite[Proposition 11.3.2(c)]{grochenig2013foundations}, \cite{wang2011harmonic}. 
\end{Remark}

\begin{Remark}\label{chr} For any given function $f$ which is locally in $B$  (Banach space) (i.e,  $gf\in B, \forall g \in C_0^{\infty}(\R^d)),$ we set $f_{B}(x)= \|f g(\cdot -x)\|_{B}.$ In \cite{Fei},  Feichtinger introduced Wiener amalgam space $W(B,C)$ with local component $B$  and global component $C$ (Banach space) is defined as the space of all functions $f$ locally in $B$ such that  $f_{B}\in C$. The space $W(B, C)$ endowed with the norm  $\|f\|_{W(B, C)}=\|f_{B}\|_{C}.$  Moreover, different choices of $g\in C^{\infty}_0(\R^d)$  generate the same space and yield equivalent norms.   We note that Fourier amalgam spaces is a Fourier image of particular Wiener amalgam spaces,  specifically,  $\mathcal{F}W(L^p, \ell^q_s)=\widehat{w}^{p,q}_s.$
\end{Remark}

The rest of the paper is organized as follows. In Section \ref{p}  recall required facts on modulation and Fourier amalgam spaces.  In Section \ref{fse} we prove Theorem \ref{dip}.

\section{Preliminaries}\label{p}  
 The notation $A \lesssim B $ means $A \leq cB$ for  some constant $c > 0 $.  The \textbf{Fourier-Lebesgue spaces} $\mathcal{F}L^p(\m)$ is defined by 
$$\mathcal{F}L^p(\m)= \left\{f\in \mathcal{S}'(\m): \|f\|_{\mathcal{F}L^{p}}:= \|\hat{f}\|_{L^{p}(\widehat{\m})}< \infty \right\}.$$
\begin{Lemma}[Basic Properties, see \cite{wang2011harmonic, grochenig2013foundations,ruzhansky2012modulation}, see  \cite{baoxiang2006isometric}  and   Corollary 2.7 in  \cite{benyi2009local}]  \label{rl} Let $p,q, p_{i}, q_{i}\in [1, \infty]$  $(i=1,2), s, s_1, s_2 \in \R.$ Then
\begin{enumerate}
\item \label{ir} $M^{p_{1}, q_{1}}_{s_1}(\mathbb R^{d}) \hookrightarrow M^{p_{2}, q_{2}}_{s_2}(\mathbb R^{d})$ whenever $p_{1}\leq p_{2}$ and $q_{1}\leq q_{2}$ and $s_2\leq s_1.$
\item \label{rcs} $M^{\min\{p', 2\}, p}(\mathbb R^d) \hookrightarrow \mathcal{F} L^{p}(\mathbb R^d)\hookrightarrow M^{\max \{p',2\},p}(\mathbb R^d),  \frac{1}{p}+\frac{1}{p'}=1.$
\item \label{ap} $M^{p,1}(\mathbb R^d)$ is an algebra under pointwise multiplication with norm inequality
$$\|fg\|_{M^{p,1}} \lesssim \|f\|_{M^{p,1}} \|g\|_{M^{p,1}}.$$
\item $\widehat{w}^{p_1, q_1}(\R^d)\subset \widehat{w}^{p_2,q_2}(\R^d)$
for $p_1\geq p_2$ and $q_1\leq  q_2$.
\item Let $\frac{1}{p_1}+ \frac{1}{p_2}=1+ \frac{1}{p}$ and $\frac{1}{q_1}+\frac{1}{q_2}=1+\frac{1}{q}.$ Then we have 
\[\|fg\|_{\widehat{w}^{p,q}}\lesssim \|f\|_{\widehat{w}^{p_1,q_1}}\|g\|_{\widehat{w}^{p_2,q_2}}.\]
In particular, 
$\widehat{w}^{p,q}(\mathcal{M})$ is an $ \mathcal{F}L^1$-module i.e. $\|fg\|_{\widehat{w}^{p,q}(\m)}\lesssim\|f\|_{\mathcal{F} L^1}\|g\|_{\widehat{w}^{p,q}(\m)}$.
\end{enumerate}
\end{Lemma}
\begin{proof}
We note that $\mathcal{F}W(L^p, \ell^q)= \widehat{w}^{p,q}.$ Since $$L^{p_1}\ast L^{p_2}\subset L^{p} \quad \text {and} \quad \ell^{q_1}\ast \ell^{q_2}\subset \ell^{q},$$
by \cite{Fei},  we have 
$
\|fg\|_{\widehat{w}^{p,q}}=\| \hat{f}\ast \hat{g}\|_{W(L^p, \ell^q)}  \lesssim \|\hat{f}\|_{W(L^{p_1}, \ell^{q_1})}\|\hat{g}\|_{W(L^{p_2}, \ell^{q_2})} \lesssim  \|f\|_{\widehat{w}^{p_1,q_1}}\|g\|_{\widehat{w}^{p_2,q_2}}.
$
\end{proof}
We refer to \cite{grochenig2013foundations}  for a classical foundation of these spaces and \cite{wang2011harmonic, ruzhansky2012modulation} for some recent developments for  nonlinear dispersive equations and the references therein. 
\subsection{Linear estimates}\label{le}
For $f\in \mathcal{S}(\m), t\in [0, \infty), 0<\beta <\infty,$ we define fractional  heat propagator  as follows
  $$U_{\beta}(t)f(x)=e^{t(-\Delta)^{\beta/2}}f(x)= \mathcal{F}^{-1} (e^{-t|\xi|^\beta}\mathcal{F}(\xi))(x)   \quad (x\in\mathcal{M},  \ \xi \in \widehat{\mathcal{M}}).$$ 
The next proposition shows that the uniform boundedness of  $U_{\beta}$ on modulation spaces. Specifically, we have following 
\begin{proposition}\label{di} Let   $1\leq p, q \leq \infty$ and $0<\beta < \infty.$ Then, for $f\in M^{p,q}(\R^d),$ we have 
\[  \|U_{\beta}(t)f\|_{M^{p,q}} \lesssim \|f\|_{M^{p,q}}.\]
\end{proposition}
\begin{proof}  We sketch the proof, for detail see \cite[Theorem 3.1]{chen2012estimates}.
By  Remark \ref{evd}, we have
\[\|U_{\beta}(t)f \|_{M^{p,q}} \lesssim  \left( \sum_{ |k|<10} \| \square_k U_{\beta}(t)f\|^q_{L^p}  \right)^{1/q} +  \left( \sum_{|k|\geq 10} \| \square_k U_{\beta}(t)f\|^q_{L^p}  \right)^{1/q}. \]
We notice almost orthogonality \cite{baoxiang2006isometric, wang2007global}  relation for the frequency-uniform decomposition operators
 \begin{eqnarray}\label{op}
 \square_k= \sum_{\|\ell \|_{\infty}\leq 1} \square_{k+\ell}\square_{k} \ \ (k, \ell \in \mathbb Z^{d})
\end{eqnarray}
where $\|\ell\|_{\infty}= \max \{|\ell_i|:\ell_i \in \mathbb Z, i=1,..., d\}.$  
By \eqref{op}, for $|k|\geq 10,$ we have 
\begin{eqnarray*}
 \square_kU_{\beta}(t)f  & =  & \sum_{\|\ell \|_{\infty}\leq 1} \square_{k+\ell}\square_{k}U_{\beta}(t)f.
\end{eqnarray*}
Using this, Young inequality  and \cite[Lemma 3.1]{chen2012estimates}  give 
\begin{eqnarray*}
\|\square_kU_{\beta}(t)f\|_{L^p} \lesssim  \sum_{|\ell|_{\infty} \leq 1}\| (\sigma_{k+\ell} e^{-t|\xi|^{\beta}} )^{\vee}\|_{L^1} \|\square_{k}f\|_{L^p} \lesssim \|\square_kf\|_{L^p}.
\end{eqnarray*} 
Exploiting  the proof of   \cite[Lemma 3.5]{chen2012estimates}, we   may obtain  $\|\square_kU_{\beta}(t)f\|_{L^p} \lesssim  \|\square_kf\|_{L^p}$ for $|k|<10.$ Combining the above inequalities, we obtain the desired inequality. 
\end{proof}
\section{The  proof of Theorem \ref{dip}}\label{fse}
In this section we prove Theorem \ref{dip}.  We start with introducing the formal solution of \eqref{fh} in the  next subsection.
\subsection{The formal solution of  fractional heat equation}{\label{fsn}}
In this subsection, we introduce the formal solution of fractional heat equation  \eqref{fh}. Assume that $k=2$ and 
\begin{eqnarray*}
u_1(t) & =  & U_{\beta}(t)u_0,\\
u_2(t) & =  & \int_0^t U_{\beta} (t-s) u_1^2 ds,\\
u_3(t) & = &  \int_0^t U_{\beta} (t-s) 2u_1 u_2 ds,\\
&&  ....,\\
u_{2n}(t) & = & \int_0^t U_{\beta} (t-s) (2u_1u_{2n-1} + \cdots + 2 u_{n-1} u_{n+1} + u_n^2) ds,\\
u_{2n+1}(t) & = & \int_0^t U_{\beta} (t-s) (2u_1u_{2n} + \cdots + 2 u_{n-1} u_{n+2} +2 u_n u_{n+1}) ds,\\
&& \cdots.
\end{eqnarray*}
Then $u= \sum_{i=1}^{\infty} u_i$ (formal solution) formally satisfies  the integral equation
\[  u(t) = U_{\beta} (t) u_0 + \int_0^t\ U_{\beta} (t-s) u^2(s) ds.\]
In fact, we have 
\begin{eqnarray*}
\sum_{i=1}^{\infty} \widehat{u_i} & = & \widehat{ u_1} + \int_0^t e^{-(t-s) |\xi|^{\beta}} \mathcal{F}(u_1^2) ds +  \int_0^{t} e^{-(t-s)|\xi|^{\beta}} \mathcal{F}(2u_1u_2) ds +...\\
&& + \int_0^t e^{-(t-s) |\xi|^{\beta}} \mathcal{F}(2u_1 u_{2n}+...+ 2u_nu_{n+1}) ds+ ...\\
& = & \widehat{u_1} +\int_0^t e^{-(t-s)|\xi|^{\beta}} \mathcal{F} (u_1^2+ 2 u_1u_2+...) ds\\
& = & \widehat{u_1} + \int_0^t e^{- (t-s) |\xi|^{\beta}} \mathcal{F} \left( \sum_{i=1}^{\infty} u_i \right)^{2} ds.
\end{eqnarray*}
By taking inverse Fourier transform on two sides, we can obtain that 
\[u(t) =U_{\beta}(t)u_{0} + \int_0^t U_{\beta}(t-s) u^2 ds.\]
\begin{Remark} The formal analysis (formal solution for \eqref{fh}) performed above will be made rigorous in the following subsections.
\end{Remark}
\begin{Lemma} For  $f(u)=u^k,$ assume 
\begin{eqnarray*}
u_1(t) & =  & U_{\beta}(t)u_0,\\
u_k(t) & =  & \int_0^t U_{\beta} (t-s) u_1^k ds,\\
u_{2k-1}(t) & = &  \int_0^t U_{\beta} (t-s) C_k^1u_1^{k-1} u_k ds,\\
u_{3k-2}(t) & = & \int_0^t U_{\beta} (t-s) (C_k^2 u_1^{k-2} u_k^2+C_k^1u_1^{k-1} u_{2k-1}) ds,\\
u_{4k-3}(t) & = & \int_0^t U_{\beta} (t-s) (C_k^{3} u_1^{k-1} u_k^3+ C_k^1u_1^{k-1} u_{3k-2}) ds,\\
&& \cdots\\
u_{jk-(j-1)} (t) & = & \int_0^t U_{\beta} (t-s)  \left( \sum_{\Lambda_j}  (u_{i_1}...u_{i_k}) \right) ds,\\
&&   ...,
\end{eqnarray*}
\end{Lemma}
where  \[\Lambda_j = \left\{ (i_1,..., i_k): i_1+\cdots + i_k= jk-(j-1), i_m=tk-(t-1), 0\leq t<j, t\in \mathbb N, m=1,...,k \right\}.\]
Then $u= \sum_{i=1}^{\infty} u_{ik-(i-1)}$ (formal solution) formally satisfies  the integral equation
\[  u(t) = U_{\beta} (t) u_0 + \int_0^t\ U_{\beta} (t-s) f(u(s)) ds.\]

\subsection{Local well-posedness in $X(\m)$}\label{lws}
\begin{proposition}[local well-posedness]\label{lw}  Assume that $u_0 \in X(\m) \ (1\leq p \leq \infty).$  Then there exists $T_1>0$ such that  \eqref{fh} has a unique solution $u \in C([0, T_1), X(\m)).$
\end{proposition}
\begin{proof} By Duhamle principle,  \eqref{fh} can be rewritten as 
\begin{equation} \label{fh1}
u(\cdot, t)= U_{\beta}(t)u_{0}+\int_{0}^{t}U_{\beta}(t-s) f(u) ds:=\mathcal{J}(u),
\end{equation}
where $U_{\beta}(t)=e^{-t (-\Delta)^{\beta/2}}.$
We show that  the mapping  $\mathcal{J}$
has a unique fixed point in an appropriate functions space, for small $t$. 
For this, we consider the Banach space $X_{T}=C([0, T],  M^{p,1}(\R^d))$, with norm
$$\left\|u\right\|_{X_{T}}=\sup_{t\in [0, T]}\left\|u(\cdot, t)\right\|_{X}, \ (u\in X_{T}).$$ 
By Minkowski's inequality for integrals and Propositions \ref{di} and  Lemma \ref{rl} \refeq{ap}, we obtain
\begin{eqnarray*}
\left\| \int_{0}^{t} U_{\beta}(t-s)  u^k(s)  \, ds \right\|_{X} 
   &\leq & T  \, \| u^k(t)\|_{X}  \leq  T\|u\|_{C([0, T], X)}^{k}.
\end{eqnarray*}
By  Proposition \ref{di}, and using above inequality, we have
\begin{eqnarray*}
\|\mathcal{J}u\|_{C([0, T], X)} \leq   \|u_{0}\|_{X} + c T \|u\|_{X}^{k},
\end{eqnarray*}
for some universal constant $c.$
 For $M>0$, put  $$B_{T, M}= \{u\in C([0, T],  X)):\|u\|_{C([0, T], M^{p,1})}\leq M \},$$  which is the  closed ball  of radius $M$, and centered at the origin in  $C([0, T],  X)$.  
Next, we show that the mapping $\mathcal{J}$ takes $B_{T, M}$ into itself for suitable choice of  $M$ and small $T>0$. Indeed, if we let, $M= 2\|u_{0}\|_{X}$ and $u\in B_{T, M},$ it follows that 
\begin{eqnarray*}
\|\mathcal{J}u\|_{C([0, T],  X)} \leq  \frac{M}{2} + cT M^{3}.
\end{eqnarray*}
We choose a  $T$  such that  $c TM^{2} \leq 1/2,$ that is, $T \leq \frac{1}{2cM^2}$ and as a consequence  we have
\begin{eqnarray*}
\|\mathcal{J}u\|_{C([0, T],  X)} \leq \frac{M}{2} + \frac{M}{2}=M,
\end{eqnarray*}
that is, $\mathcal{J}u \in B_{T, M}.$
By  Proposition \ref{ap}, and the arguments as before, we obtain
\begin{eqnarray*}
\|\mathcal{J}u- \mathcal{J}v\|_{C([0, T],  X)} \leq \frac{1}{2} \|u-v\|_{C([0, T],  X)}.
\end{eqnarray*}
Therefore, using the  Banach's contraction mapping principle, we conclude that $\mathcal{J}$ has a fixed point in $B_{T, M}$ which is a solution of \eqref{fh1}. 
\end{proof}

\subsection{Finite time blow-up in $X(\m)$}\label{ftb}
In this subsection, we prove Theorem \ref{dip}. To this end, we start with flowing technical lemmas.
\begin{Lemma}\label{ia}  There exists $\epsilon_0$ such that   if  $T_2 \|u_1\|_{L^{\infty}_{T_2}(X)} < T_2 \|u_0\|_{X} < \epsilon_0$ then \[\sum_{i=1}^{\infty} \|u_{i}\|_{L^{\infty}_{T_2}(X)} < \infty.\]
\end{Lemma}
\begin{proof}
 Taking notations of Subsection \ref{fsn} and  Proposition \ref{di} and Lemma \ref{rl} \eqref{ap}  into account, we have 
 \begin{eqnarray*}
 \|u_1\|_{L^{\infty}_{T_2}(X)}  = \|U_{\beta}(\cdot)u_0\|_{L^{\infty}_{T_2}(X)} \leq C \|u_0\|_{X},   
 \end{eqnarray*}
\begin{eqnarray*}
\|u_2\|_{L^{\infty}_{T_2}(X)} 
& \leq & C_{0} \left\| \int_0^t  \|u_1 \|^2_{X} ds\right\|_{L^{\infty}_{T_2}} \leq  C_0 T_2  \|u_1\|^2_{L^{\infty}_{T_2}(X)},
\end{eqnarray*} 
\begin{eqnarray*}
\|u_3\|_{L^{\infty}_{T_2}(X)} 
& \leq & 2 C_0 T_2  \|u_1\|_{L^{\infty}_{T_2}(X)}  \|u_2\|_{L^{\infty}_{T_2}(X)},
\end{eqnarray*} 
\begin{eqnarray*}
\|u_4\|_{L^{\infty}_{T_2}(X)} 
& \leq & 2C_0 T_2  \|u_1\|_{L^{\infty}_{T_2}(X)}  \|u_3\|_{L^{\infty}_{T_2}(X)} + C_0 T_2  \|u_2\|^2_{L^{\infty}_{T_2}(X)},
\end{eqnarray*}
 \[\cdots.\]
Notice that  the norm $\|u_{2}\|_{L^{\infty}_{T_2}(X)}$ can be controlled by   the norm  $\|u_{1}\|_{L^{\infty}_{T_2}(X)}$ and so the norms of $\|u_{i}\|_{L^{\infty}_{T_2}(X)}$ also can be controlled by $\|u_{1}\|_{L^{\infty}_{T_2}(X)}$. In view of this,  we have 
\[\sum_{i=1}^{\infty} \|u_{i}\|_{L^{\infty}_{T_2}(X)} \leq S \]
where
\begin{multline*}  
 S:=\|u_1\|_{L^{\infty}_{T_2}(X)} + C_0 T_2 \| u_1\|^2_{L^{\infty}_{T_2}(X)} + C_0 T_2  \|u_1\|_{L^{\infty}_{T_2}(X)}  \|u_2\|_{L^{\infty}_{T_2}(X)}+ \cdots.
 \end{multline*}
Next we  claim that there exists $\epsilon_0$ such that if the initial data $T_2 \|u_0\|_{X} < \epsilon_0,$ then  $S<\infty.$
  To justify the claim, we  let
\[ C= \frac{1}{2C_0T_2},\quad a_1 =\|u_1(t)\|_{L^{\infty}_{T_2}(X)}, \quad  a_{2i} = \sum_{j=1}^{i} a_{j} a_{2i-j}\quad \text{and} \quad a_{2i+1} = \sum_{j=1}^{i} a_{j} a_{2i+1-j}. \] 
 Then  we  have 
 \begin{eqnarray*}
S \leq  \sum_{m} \frac{1}{C^{m-1}} a_m \sim \sum_{m} \frac{1}{C^{m}} a_m. 
 \end{eqnarray*}
 Then, by induction, we can obtain that there exists $\epsilon_0$ such that if $T_2 \|u_0\|_{X} < \epsilon_0,$ then   $\sum_{i=1}^{\infty} \frac{a_i}{C^i} <\sum_{i=1}^{\infty} \frac{1}{i^{1+ \epsilon}}< \infty$ ($\epsilon$ to be decided). Indeed, if for any $m\leq 2i, a_m< \frac{C^m}{m^{1+\epsilon}},$ then we have
\begin{eqnarray*}
 && a_{2i+1} \\
&  \leq &  a_1 \frac{C^{2i}}{(2i)^{1+\epsilon}} + a_2 \frac{C^{2i-1}}{(2i-1)^{1+ \epsilon}} + \cdots + \frac{C^{2i+1}}{(i(i+1))^{1+\epsilon}}\\
  & \leq  & a_1 \frac{C^{2i+1}}{ C(2i)^{1+\epsilon}} + a_2 \frac{C^{2i+1}}{ C^2(2i-1)^{1+ \epsilon}}+ \sum_{j=3}^{i}  \frac{C^{2i+1}}{(j(2i+1-j))^{1+\epsilon}}\\
 & \leq  & \frac{a_1 (2i+1)^{1+\epsilon}}{ C(2i)^{1+\epsilon}} \frac{C^{2i+1}}{(2i+1)^{1+\epsilon}}+  \frac{a_2 (2i+1)^{1+\epsilon}}{ C^2(2i-1)^{1+ \epsilon}} \frac{C^{2i+1}}{(2i+1)^{1+\epsilon}} + \sum_{j=3}^{i}  \frac{(2i+1)^{1+\epsilon}}{(j(2i+1-j))^{1+\epsilon}} \frac{C^{2i+1}}{(2i+1)^{1+\epsilon}}\\
 & <  & \frac{C^{2i+1}}{ (2i+1)^{1+\epsilon}}.
\end{eqnarray*} 
 Choose $\epsilon>0$ so   $ \frac{a_1 (2i+1)^{1+\epsilon}}{ C(2i)^{1+\epsilon}} <1/3,  \frac{a_2 (2i+1)^{1+\epsilon}}{ C^2(2i-1)^{1+ \epsilon}}<1/3 $ and $ \sum_{j=3}^{i}  \frac{(2i+1)^{1+\epsilon}}{(j(2i+1-j))^{1+\epsilon}} < 1/3$ and hence we may obtain the claim.   Taking these observation into account,  there exits $\epsilon_0$ such that if $T_2 \|u_0\|_{X} < \epsilon_0,$ then  $u(t)= \sum_{i=1}^{\infty} u_i(t) \in L^{\infty} ([0, T_2),  X(\m)).$
\end{proof}
\begin{Lemma}\label{me1} Let $T_1$ and $T_2$ be as in Proposition \ref{lw} and Lemma \ref{ia}.  For $T= \min\{ T_1, T_2 \}, u(t)= \sum_{i=1}^{\infty} u_i(t) \in L^{\infty}([0, T), X(\m))$ satisfies the integral equation
\[u(t) =U_{\beta}(t)u + \int_0^t U_{\beta}(t-s) u^2 ds.\]
\end{Lemma}
\begin{proof} We claim that 
\begin{equation}\label{de1}
 \lim_{M\to \infty}  \int_0^t U_{\beta}(t-s) \left( \left( \sum_{i=1}^{M}  u_i\right) \left( \sum_{i=1}^{M}   u_i\right)  \right) ds = \int_0^t U_{\beta} (t-s) \left( \left( \sum_{i\in \mathbb N}  u_i\right) \left( \sum_{i\in \mathbb N}   u_i\right) \right) ds
\end{equation}
in  $X_T:= L^{\infty} ([0, T),  X(\m)).$
By the definition of $u_i$ in the formal solution and $\widehat{u_0}\geq 0,$ we have 
\begin{eqnarray*}
 D & := &  \left\|\int_0^t U_{\beta} (t-s) \left( \left( \sum_{i\in \mathbb N}  u_i\right) \left( \sum_{i\in \mathbb N}   u_i\right) \right) ds -  \int_0^t U_{\beta}(t-s) \left(  \left( \sum_{i=1}^{M}  u_i\right) \left( \sum_{i=1}^{M}   u_i\right) \right) ds \right\|_{X_T}\\
 & \leq &   \left\|\int_0^t U_{\beta} (t-s) \left( \left( \sum_{i\in \mathbb N}  u_i\right) \left( \sum_{i\in \mathbb N}   u_i\right) \right) ds -  \int_0^t U_{\beta}(t-s)\left(  \left( \sum_{i=1}^{M}  u_i\right) \left( \sum_{i\in \mathbb N}  u_i\right)  \right) ds   \right\|_{X_T}\\
 && +  \left\| \int_0^t U_{\beta} (t-s) \left(  \left( \sum_{i=1}^{M} u_i \right)  \left(  \sum_{i\in \mathbb N} u_i\right)  \right) ds -\int_0^t U_{\beta} (t-s) \left(  \left( \sum_{i=1}^{M} u_i \right)  \left(  \sum_{i=1}^M u_i\right)    \right) ds \right\|_{X_T}\\
 & \lesssim & \left\| \int_0^t U_{\beta} (t-s) \left(  \left( \sum_{i=M+1}^{\infty} u_i \right)  \left(  \sum_{i\in \mathbb N} u_i\right)   \right) ds \right\|_{X_T}\\
 && + \left\| \int_0^t U_{\beta} (t-s) \left(  \left( \sum_{i=1}^{M} u_i \right)  \left(  \sum_{i=M+1}^{\infty} u_i\right)    \right) ds  \right\|_{X_T}\\
 & \leq & 2 C_0 \left\| \sum_{i=M+1}^{\infty} u_i \right\|_{X_T}.
\end{eqnarray*}
In view of this and Lemma \ref{ia}, we have that  for any small $\epsilon>0,$ there exist $N$ such  that for any $M>N,  D< \epsilon.$ This proves \eqref{de1}.
Next, we show that 
\begin{equation}\label{de2}
 \lim_{M \to \infty} \int_0^t U_{\beta} (t-s)   \left( \sum_{i=1}^{M} u_i \right)   \left( \sum_{i=1}^{M} u_i \right)   ds =  \lim_{M \to \infty} \int_0^t U_{\beta} (t-s)   \left( \sum_{i=2}^{M} \sum_{j=1}^{i-1} u_j u_{(i-j)} \right) ds
\end{equation}
in $X_T.$  Indeed, by Lemma \ref{ia} and  the  similar argument as above, we have 
\begin{eqnarray*}
E & :=  & \left\|  \int_0^t U_{\beta} (t-s)  \left(  \left( \sum_{i=1}^{M} u_i \right)   \left( \sum_{i=1}^{M} u_i \right) \right)   ds-   \int_0^t U_{\beta} (t-s)   \left( \sum_{i=2}^{M} \sum_{j=1}^{i-1} u_j u_{(i-j)} \right)    ds \right\|_{X_T}\\
& \leq & \left\|  \int_0^t U_{\beta} (t-s) \left(  \left( \sum_{i=1}^{M} u_i \right)   \left( \sum_{i=1}^{M} u_i \right) \right)   ds-   \int_0^t U_{\beta} (t-s)   \left( \left( \sum_{i=1}^{[\frac{M}{2}]} u_i \right)   \left( \sum_{i=1}^{[\frac{M}{2}]} u_i \right) \right)   ds\right\|_{X_T}\\
&& + \left\| \int_0^t U_{\beta} (t-s)  \left( \left( \sum_{i=1}^{[\frac{M}{2}]} u_i \right)   \left( \sum_{i=1}^{[\frac{M}{2}]} u_i \right) \right)   ds - \int_0^t U_{\beta} (t-s)   \left( \sum_{i=2}^{M} \sum_{j=1}^{i-1} u_j u_{(i-j)} \right)    ds \right\|_{X_T}\\
& \leq & 2C_0  \left\| \sum_{i= [\frac{M}{2}]+1}^{2M} u_i \right\|_{X_T} + C_0 \sum_{i= [\frac{M}{2}] +1}^{2M} \|u_i\|_{X_T}.
\end{eqnarray*} From this  \eqref{de2} follows.
Taking \eqref{de1}, \eqref{de2} and  definition of $u_i$ in the formal solution (Subsection \ref{fsn}) into account, we obtain
\begin{eqnarray*}
\sum_{i \in \N} u_i & = & U_{\beta} (t) u_0 + \lim_{M\to \infty} \sum_{i=2}^{M} u_i =    U_{\beta} (t) u_0 + \lim_{M\to \infty} \sum_{i=2}^{M}  \int_0^t U_{\beta} (t-s) \left( \sum_{q=1}^{i-1} u_q u_{i-q} \right) ds\\
& = & U_{\beta} (t) u_0 + \lim_{M\to \infty}   \int_0^t U_{\beta} (t-s) \left( \sum_{i=2}^{M}  \sum_{q=1}^{i-1} u_q u_{i-q} \right) ds
\end{eqnarray*}
in $X_T.$ By the above argument, we can obtain that 
\[ u(t) = \sum_{i=1}^{\infty} u_i(t) \in L^{\infty} ([0, T), X(\m)) \]
is a solution of \eqref{fh}. 
 \end{proof}
Let   $T$ be as in Lemma \ref{me1} and  $u(T)$ be the initial data, by repeating  previous process (Proposition \ref{lw} and Lemmas \ref{ia} and \ref{me1}), we can obtain  $t_2$ such that $u(t) = \sum_{j=1} a_j \in L^{\infty} ([T, t_2),  X(\m))$ is a unique solution of the integral equation
\[u(t) =U_{\beta}(t)u + \int_0^t U_{\beta}(t-s) u^2 ds,\]
where, for $T<t< t_2,$
\begin{eqnarray*}
a_1(t) & = & U_{\beta} (t-T) u_T,\\
a_2(t) & = &  \int_T^{t} U_{\beta}(t-s) a_1^2 ds,\\
a_3(t) & = &  \int_T^{t} U_{\beta}(t-s)  2(a_1 a_2) ds,\\
a_4(t) & = &  \int_T^{t} U_{\beta}(t-s)  (2 a_1 a_3 + a_2^2) ds,\\
a_5(t) & = &  \int_T^{t} U_{\beta}(t-s)  (2 a_1 a_4 + 2 a_2 a_3) ds,\\
&& \cdots,\\
a_{2m}(t) & = &  \int_T^{t} U_{\beta}(t-s)  (2 a_1 a_{2m-1} + \cdots + 2a_{m-1} a_{m+1}+ a_m^2) ds,\\
a_{2m+1}(t) & = &  \int_T^{t} U_{\beta}(t-s)  (2 a_1 a_{2m} + \cdots + 2a_{m-1} a_{m+2} + 2 a_{m} a_{m+1}) ds,\\
&&\cdots .
\end{eqnarray*}
Step by step we can obtain a  sequence $\{ t_j \}_{j=0}^{\infty}, t_0=0, t_1=T$ and $t_j \to T^*$ such that for every  $0 \leq  i < \infty$ and   $t_i< t< t_{i+1}, u(t)= \sum_{j=1} a_j(t) \in L^{\infty}([t_i, t_{i+1}),  X(\m))$ is the unique solution of integral equation
\[u(t) =U_{\beta}(t-t_i)u(t_i) + \int_{t_i}^{t_{i+1}} U_{\beta}(t-s) u^2 ds,\]
where, for $t_i<t< t_{i+1},$
\begin{eqnarray*}
a_1(t) & = & U_{\beta} (t-t_i) u(t_i),\\
a_2(t) & = &  \int_{t_i}^{t} U_{\beta}(t-s) a_1^2 ds,\\
a_3(t) & = &  \int_{t_i}^{t} U_{\beta}(t-s)  2(a_1 a_2) ds,\\
a_4(t) & = &  \int_{t_i}^{t} U_{\beta}(t-s)  (2 a_1 a_3 + a_2^2) ds,\\
a_5(t) & = &  \int_{t_i}^{t} U_{\beta}(t-s)  (2 a_1 a_4 + 2 a_2 a_3) ds,\\
&& \cdots,\\
a_{2m}(t) & = &  \int_{t_i}^{t} U_{\beta}(t-s)  (2 a_1 a_{2m-1} + \cdots + 2a_{m-1} a_{m+1} + a_m^2) ds,\\
a_{2m+1}(t) & = &  \int_{t_i}^{t} U_{\beta}(t-s)  (2 a_1 a_{2m} + \cdots +2 a_{m-1} a_{m+2} + 2 a_{m} a_{m+1}) ds,\\
&&\cdots .
\end{eqnarray*}
In this way, we can  obtain $ u(t)= \sum_{j=1} a_j(t) \in L^{\infty}([0,  T^*), X(\m) )$ is the unique solution of integral equation
\begin{eqnarray}\label{fh2}
u(t) =U_{\beta}(t)u + \int_{0}^{t} U_{\beta}(t-s) u^2 ds.
\end{eqnarray}
Moreover, if $T^*< \infty,$ then 
\[ \| u(t)\|_{L^{\infty}([0, T^*),  X)} = \infty.\]
By the similar argument as  above we  can obtain the corresponding  results for $f(u)=u^k.$

\begin{Lemma}\label{cl} Let $u_0\in X(\m)$ and $\widehat{u_0}\geq 0.$ Then
\[\sum_{j=1}^{\infty} \widehat{a_j}(t) \geq \sum_{j=1}^{\infty} \widehat{u_j}(t),\]
where  
\begin{eqnarray*}
u_1(t) & = & U_{\beta}(t) u_0, \quad u_2(t)  =  \int_{0}^t U_{\beta} (t-s) u_1^2 ds,\\
u_3(t) & = & \int_{0}^t U_{\beta} (t-s)  2u_1 u_2 ds = \int_0^{t} U_{\beta} (t-s) \sum_{\Lambda_3} u_{i_1} u_{i_2} ds, \cdots,\\
u_{2n}(t) & = & \int_{0}^t U_{\beta} (t-s) (2 u_1u_{2n-1} + \cdots + 2 u_{n-1} u_{n+1} + u_n^2) ds 
=   \int_0^{t} U_{\beta} (t-s) \sum_{\Lambda_{2n}} u_{i_1} u_{i_2} ds,
\end{eqnarray*}
\[ u_{2n+1}(t)  =  \int_{0}^t U_{\beta} (t-s) (2 u_1u_{2n} + \cdots + 2 u_{n-1} u_{n} + 2 u_n u_{n+1}) ds 
 =   \int_0^{t} U_{\beta} (t-s) \sum_{\Lambda_{2n+1}} u_{i_1} u_{i_2} ds,\]
..., and $\Lambda_j = \{ (i_1, i_2) : i_1 + i_2 =j,  i_m \in \N, m=1,2 \}.$ Actually, by the fact that  for $0\leq  t < T^*, \sum_{j=1} \widehat{ a_j} \in L^1(\widehat{\m}),$ we have $\sum_{j=1} \widehat{a_j} = \mathcal{F} \left( \sum_{j=1} a_j(t) \right).$ So, $\sum_{j=1} \widehat{u_j(t)}$ and $\sum_{j=1} \widehat{a_j(t)}$ are non negative  term series.  Moreover, by  $\{ \widehat{u_j(t)}\}_{j=1}^{\infty}$ is the rearrangement of  $\{ \widehat{a_j(t)}\}_{j=1}^{\infty}$.
\end{Lemma}
\begin{proof}
  We just prove that $\sum_{j=1}^{\infty} \widehat{u_j}\leq  \sum_{j=1}^{\infty} \widehat{a_j}.$ For $t_1=T<t< t_2,$ by  $\sum_{j=1}^{\infty} a_j(t) \in L^{\infty}([0, T), X(\m))$ satisfies the integral heat equation \eqref{fh2}, we have 
  
\begin{eqnarray*}
\widehat{a_1}(t) & = & e^{-(t-T) |\xi|^{\beta}} \widehat{u}(T) =  e^{-(t-T) |\xi|^{\beta}} \mathcal{F} \left(U_{\beta}(T) u_0 +  \int _0^T U_{\beta}(T-s) u^2 (s) ds \right)\\
& = &   e^{-t|\xi|^{\beta}} \widehat{u_0} +  \int _0^T  e^{-(t-s)|\xi|^{\beta}} \left(  \sum_{j=1}^{\infty}  \widehat{a_j}(s)\right) \ast \left(  \sum_{j=1}^{\infty}  \widehat{a_j}(s)\right)  ds\\
& = &  e^{-t|\xi|^{\beta}} \widehat{u_0} +   \sum_{j=1}^{\infty} \int _0^T  e^{-(t-s)|\xi|^{\beta}}  \sum_{\Lambda_j} \widehat{ a_{i_1}} \ast \widehat{ a_{i_2}}  ds,
\end{eqnarray*}  
  
  \begin{eqnarray*}
  \widehat{a_2}(t)  =  \int_{T}^t e^{-(t-s)|\xi|^{\beta}} (\widehat{a_1} \ast  \widehat{a_1}) ds, \cdots,  \widehat{a_i}(t)  =  \int_{T}^t e^{-(t-s)|\xi|^{\beta}}  \sum_{\Lambda_i}\widehat{a_{i_1}} \ast  \widehat{a_{i_2}} ds, \cdots .
  \end{eqnarray*}
Moreover, by $\widehat{ a_i}(t) = \widehat{u_i}(t)$ for $0 \leq t \le T,$ we  have 
\begin{eqnarray*}
\widehat{u_1}(t) & = &  e^{-t|\xi|^{\beta} }\widehat{u_0},
\end{eqnarray*}  
\begin{eqnarray*}
\widehat{u_2}(t) & = &  \int_0^t e^{- (t-s) |\xi|^{\beta}} (\widehat{ u_1} \ast \widehat{u_1}) ds
 =  \int_0^T e^{- (t-s) |\xi|^{\beta}} (\widehat{ u_1} \ast \widehat{u_1}) ds + \int_T^t e^{- (t-s) |\xi|^{\beta}} (\widehat{ u_1} \ast \widehat{u_1}) ds\\
& = & \int_0^T e^{- (t-s) |\xi|^{\beta}} (\widehat{ a_1} \ast \widehat{a_1}) ds + \int_T^t e^{- (t-s) |\xi|^{\beta}} (\widehat{ u_1} \ast \widehat{u_1}) ds,
\end{eqnarray*} 
 \begin{eqnarray*}
\widehat{u_3}(t) & = &  \int_0^t e^{- (t-s) |\xi|^{\beta}} 2 (\widehat{ u_1} \ast \widehat{u_2}) ds
=  \int_0^T e^{- (t-s) |\xi|^{\beta}}2 (\widehat{ u_1} \ast \widehat{u_2}) ds + \int_T^t e^{- (t-s) |\xi|^{\beta}}2 (\widehat{ u_1} \ast \widehat{u_2}) ds\\
& = & \int_0^T e^{- (t-s) |\xi|^{\beta}} 2(\widehat{ a_1} \ast \widehat{a_2}) ds + \int_T^t e^{- (t-s) |\xi|^{\beta}} 2(\widehat{ u_1} \ast  \int_0^T e^{-(s-s_1)|\xi|^{\beta}} (\widehat{a_1}\ast \widehat{a_1}) ds_1) ds\\
&&  + \int_T^t e^{- (t-s) |\xi|^{\beta}} 2(\widehat{ u_1} \ast  \int_T^t e^{-(s-s_1)|\xi|^{\beta}} (\widehat{u_1}\ast \widehat{u_1}) ds_1) ds,\\
&& \cdots.
\end{eqnarray*}   
Note that 
$\int_0^T e^{- (t-s) |\xi|^{\beta}} 2 (\widehat{ a_1} \ast \widehat{a_1}) ds$ and $\int_0^T e^{- (t-s) |\xi|^{\beta}} 2 (\widehat{ a_1} \ast \widehat{a_2}) ds  $  are expansion  term of $\widehat{a_1}$, 
$\int_T^t e^{- (t-s) |\xi|^{\beta}} (\widehat{ u_1} \ast \widehat{u_1}) ds $ and $\int_T^t e^{- (t-s) |\xi|^{\beta}} 2(\widehat{ u_1} \ast  \int_0^T e^{-(s-s_1)|\xi|^{\beta}} (\widehat{u_1}\ast \widehat{u_1}) ds_1) ds$   are the expansion term of  $ \widehat{a_2},$...
Then by induction
\begin{eqnarray*}
u_{2n}(t)  =  \int_0^t U_{\beta} (t-s) (2 u_1 u_{2n-1} + ...+ 2 u_{n-1}u_{n+1} + u_n^2) ds =  \int_0^t U_{\beta} (t-s) \sum_{\Lambda_{2n}} u_{i_1} u_{i_2} ds,\\
u_{2n+1}(t)  =  \int_0^t U_{\beta} (t-s) (2 u_1 u_{2n} + ...+ 2 u_{n-1}u_{n} + u_n u_{n+1}) ds
 =  \int_0^t U_{\beta} (t-s) \sum_{\Lambda_{2n+1}} u_{i_1} u_{i_2} ds,
\end{eqnarray*}
..., and comparing  the expansion  term of $\widehat{a_i}$ and $\widehat{ u_i}$ (by splitting  the integral  $\int_0^t$ into  $\int_0^T$ and $\int_{T}^t$), it is easy to see that every expansion term of $\widehat{u_i}$ is a expansion  term of $\widehat{a_i}.$ (Here, the every expansion term is  non negative.) Then, by induction, we have $\sum_{j=1}  \widehat{ u_j} \leq \sum_{j=1}^{\infty} \widehat{a_j} \in L^1(\widehat{\m})$ for $0<t< T^*.$
We have  $\mathcal{F}\left(\sum_{j=1}^{\infty} u_j \right) = \sum_{j=1}^{\infty} \widehat{ u_j}$ and 
\begin{eqnarray*}
\left\|  \sum_{j=1}^{\infty} u_j \right\|_{L^{\infty}_{loc} ([0, T^*), X)} \leq \left\|  \sum_{j=1}^{\infty} a_j \right\|_{L^{\infty}_{loc} ([0, T^*),  X)}
\end{eqnarray*}
Moreover, by the dominated convergence theorem, we have 
\begin{eqnarray*}
\left\|  \sum_{j=1}^{\infty} u_j \right\|_{L^{\infty}_{loc} ([0, T^*), X)} \leq \left\|  \sum_{j=1}^{\infty} \|u_j\|_{X} \right\|_{L^{\infty}_{loc} [0, T^*)} .
\end{eqnarray*}
\end{proof}
The following lemma will play crucial role to prove Theorem \ref{dip}.
\begin{Lemma}\label{fl} There exists  $0<T< \infty$ such that
\[  \left\| \sum_{j=1}^{\infty} u_j \right\|_{L^{\infty}([0, T),  X)} = \infty.\] 
\end{Lemma}
\begin{proof}
Let 
$$ E_{x,y} = \left \{ x, y \in \mathbb R^d: |x|^{\beta} \leq |x-y|^{\beta} + |y|^{\beta} \right\},$$
$$F_{x,y} = \left \{ x, y \in \mathbb R^d: |x|^{\beta} \geq |x-y|^{\beta} + |y|^{\beta} \right\} .$$ By $\widehat{u_0}\geq 0, \widehat{u_0} \geq \gamma \chi_{B_0(r)}$ and   $r^d v_d \geq 2^d$ (which implies  that   $\chi_{B_0(r)}\ast \chi_{B_0(r)}(\xi)  \geq \chi_{B_0(r)}(\xi)$), we have
\begin{eqnarray}\label{1}
\widehat{u_1}(t, \xi) & =  &e^{-t|\xi|^{\beta}} \widehat{u_0}(\xi) \geq e^{-t|\xi|^{\beta}} \gamma \chi_{B_0(r)}(\xi)= a_{1}(t, \xi),
\end{eqnarray}
\begin{eqnarray*}
\widehat{u_2}(t, \xi) &  = &  \int_0^t e^{-(t-s)|\xi|^{\beta}}\mathcal{F}(u_1^2)(s, \xi) ds
  =   \int_0^t e^{-(t-s)|\xi|^{\beta}}  \left((e^{-s|\cdot|^{\beta} }\widehat{u_0})\ast (e^{-s|\cdot|^{\beta} }\widehat{u_0})\right)(\xi) ds\\
& = &  \int_0^t e^{-(t-s)|\xi|^{\beta}}  \int_{\R^d}  e^{-s |\xi-y|^{\beta}} e^{-s|y|^{\beta}} \widehat{u_0}(\xi-y) \widehat{u_0}(y) dy ds\\
& \geq &  \gamma^2 \int_0^t e^{-(t-s)|\xi|^{\beta}}  \int_{\R^d}  e^{-s |\xi-y|^2} e^{-s|y|^{\beta}} \chi_{B_0(r)}(\xi-y) \chi_{B_0(r)}(y) dy ds\\
& = &   \gamma^2 \int_0^t e^{-(t-s)|\xi|^{\beta}}  \int_{\R^d}  e^{-s |\xi-y|^{\beta}} e^{-s|y|^{\beta}}\\
&&  \left( \chi_{B_0(r)}(\xi-y) \chi_{B_0(r)}(y)\mid_{E_{\xi,y}} + \chi_{B_0(r)}(\xi-y) \chi_{B_0(r)}(y)\mid_{F_{\xi,y}}  \right) dy ds.\\
\end{eqnarray*}
Notice that for  $\xi, y \in E_{x,y}$ and $0<s<t,$ we have  $e^{-(t-s) |\xi|^{\beta}} e^{-s|\xi-y|^{\beta}} e^{-s|y|^{\beta}} \gtrsim e^{-t|\xi-y|^{\beta}} e^{-t|y|^{\beta}}.$  Similarly, for $\xi, y \in F_{x,y},$ we have   $e^{-(t-s) |\xi|^{\beta}} e^{-s|\xi-y|^{\beta}} e^{-s|y|^{\beta}} \gtrsim e ^{-t|\xi|^{\beta}}.$
\begin{eqnarray}\label{2}
 \widehat{u_2}(t, \xi) 
& \gtrsim &  \gamma^2 \int_0^t  \int_{\R^d} e^{-t|\xi-y|^{\beta}} e^{-t|y|^{\beta}}  \chi_{B_0(r)}(\xi-y) \chi_{B_0(r)}(y)\mid_{E_{\xi,y}}  dy ds \nonumber \\
&& + \gamma^2 e^{-t |\xi|^{\beta}} \int_0^t  \int_{\mathbb R^d}  \chi_{B_0(r)}(\xi-y) \chi_{B_0(r)}(y)\mid_{F_{\xi,y}}  dy ds \nonumber \\
& = &  \gamma^2 t  \int_{\R^d} e^{-t|\xi-y|^{\beta}} e^{-t|y|^{\beta}}  \chi_{B_0(r)}(\xi-y) \chi_{B_0(r)}(y)\mid_{E_{\xi,y}}  dy \nonumber \\
&& + \gamma^2 t e^{-t |\xi|^{\beta}}  \int_{\mathbb R^d}  \chi_{B_0(r)}(\xi-y) \chi_{B_0(r)}(y)\mid_{F_{\xi,y}}  dy \nonumber \\
& \gtrsim &  \gamma^2 t e^{-2r^{\beta} t}  e^{-t |\xi|^{\beta}} \int_{\R^d \cap  E_{\xi, y }}   \chi_{B_0(r)}(\xi-y) \chi_{B_0(r)}(y) dy \nonumber \\
&& + \gamma^2 t e^{-t |\xi|^{\beta}}  \int_{\mathbb R^d \cap F_{\xi, y}}  \chi_{B_0(r)}(\xi-y) \chi_{B_0(r)}(y)  dy  \nonumber \\
& \gtrsim &  \gamma^2 t e^{-2r^{\beta} t}  e^{-t |\xi|^{\beta}} [\chi_{B_0(r)}\ast \chi_{B_0(r)}] (\xi) \\
& \gtrsim &  \gamma^2 t e^{-2r^{\beta} t}  e^{-t |\xi|^{\beta}}  \chi_{B_0(r)}(\xi)  =  a_2(t,\xi),\nonumber
\end{eqnarray}
By  \eqref{1} and \eqref{2}, we have 
\begin{eqnarray*}
\widehat{u_3}(t, \xi) &  = & 2 \int_0^t e^{-(t-s)|\xi|^{\beta}}  \left(\widehat{u_1} \ast \widehat{u_2}\right)(\xi) ds
 \geq  2 \gamma \int_0^t e^{-(t-s)|\xi|^{\beta}}  \left(  (e^{-s|\xi|^{\beta} } \chi_{B_0(r)})\ast (\widehat{u_2}) (s, \xi)\right) ds\\
& =  &  2 \gamma \int_0^t e^{-(t-s)|\xi|^{\beta}}  \int_{\R^d}   e^{-s|\xi -\xi_1|^{\beta} } \chi_{B_0(r)} (\xi -\xi_1)\widehat{u_2} (s, \xi_1) d\xi_1  ds\\
& \gtrsim  &   \gamma^3 e^{-2r^{\beta}t} \int_0^t e^{-(t-s)|\xi|^{\beta}} s  \int_{\R^d}   e^{-s|\xi -\xi_1|^{\beta} } \chi_{B_0(r)} (\xi -\xi_1) e^{-s|\xi_1|^{\beta}} [\chi_{B_0(r)}\ast \chi_{B_0(r)}] (\xi_1)d\xi_1  ds\\
\end{eqnarray*}
Dividing the integral into two parts and arguing as before, we obtain
\begin{eqnarray*}
\widehat{u_3}(t, \xi) 
& \gtrsim &  \gamma^3 e^{-4r^{\beta}t} \int_0^t  s  \{ e^{-t|\xi|^{\beta}}  \int_{\R^d \cap F_{\xi, \xi_1}}    \chi_{B_0(r)} (\xi -\xi_1)  \chi_{B_0(r)}(\xi_1)d\xi_1  \\
&& + \int_{\R^d \cap E_{\xi, \xi_1}}   e^{-t |\xi -\xi_1|^{\beta} } \chi_{B_0(r)} (\xi -\xi_1)  e^{-t|\xi_1|^{\beta}} \chi_{B_0(r)}(\xi_1) d\xi_1\} ds\\
& = &  \gamma^3 e^{-4r^{\beta}t} t^2 \{ e^{-t|\xi|^{\beta}}  \int_{\R^d \cap F_{\xi, \xi_1}}  \chi_{B_0(r)} (\xi -\xi_1)    \chi_{B_0(r)}(\xi_1) d\xi_1 \\
&& + \int_{\R^d \cap E_{\xi, \xi_1}}   e^{-t |\xi -\xi_1|^{\beta} } \chi_{B_0(r)} (\xi -\xi_1)   e^{-t|\xi_1|^{\beta}} \chi_{B_0(r)}(\xi_1) d\xi_1\}   
\end{eqnarray*}
\begin{eqnarray*}
\widehat{u_3}(t, \xi) 
& \gtrsim  & \gamma^3 e^{-4r^{\beta}t} t^2 e^{-t|\xi|^{\beta}}  \int_{\R^d \cap F_{\xi, \xi_1}}   \chi_{B_0(r)} (\xi -\xi_1) \chi_{B_0(r)}(\xi_1) d\xi_1  \\
&& +  \gamma^3 e^{-8r^{\beta}t} t^2 e^{-t|\xi|^{\beta}}   \int_{\R^d \cap E_{\xi, \xi_1}}    \chi_{B_0(r)} (\xi -\xi_1)  \chi_{B_0(r)}(\xi_1) d\xi_1 \\
& \gtrsim & \gamma^3 e^{-8r^{\beta}t} t^2 e^{-t|\xi|^{\beta}}  [\chi_{B_0(r)} \ast \chi_{B_0(r)}](\xi) \\
& \gtrsim &  \gamma^3 e^{-8r^{\beta}t} t^2 e^{-t|\xi|^{\beta}} \chi_{B_0(r)} (\xi)\\
& = & a_3(t, \xi),\\
&& \cdots ,
\end{eqnarray*}
\begin{eqnarray*}
\widehat{u_{2n}}(t, \xi)  
& = &  \int_0^t  e^{-(t-s)|\xi|^{\beta}}\mathcal{F}(2u_1 u_{2n-1} + \cdots  + 2 u_{n-1}u_{n+1} + u_n^2)(\xi) ds\\
& \geq & \int_0^t e^{-(t-s)|\xi|^{\beta}} (2 a_1 (s, \xi) \ast a_{2n-1}(s, \xi)+\cdots \\
&& + 2 a_{n-1} (s, \xi) \ast a_{n+1}(s, \xi) + a_n(s, \xi) \ast a_n(s, \xi)) ds\\
&= & a_{2n} (t, \xi),
\end{eqnarray*}
\begin{eqnarray*}
\widehat{u_{2n+1}}(t, \xi) 
& = &  \int_0^t  e^{-(t-s)|\xi|^{\beta}}\mathcal{F}(2u_1 u_{2n} + \cdots +  2 u_{n-1}u_{n+2} +2 u_n u_{n+1})(\xi) ds\\
& \geq & \int_0^t e^{-(t-s)|\xi|^{\beta}} (2 a_1 (s, \xi) \ast a_{2n}(s, \xi)+\cdots \\
&& + 2 a_{n-1} (s, \xi) + a_{n+2}(s, \xi) \ast 2 a_n(s, \xi) \ast a_{n+1}(s, \xi)) ds\\
&= & a_{2n+1} (t, \xi),\\
&& \cdots .
\end{eqnarray*}
Firstly, we have  $a_1(t, \xi) = \gamma e^{-t |\xi|^{\beta}} \chi_{B_0(r)} (\xi).$ Secondly, if 
\begin{eqnarray*}
 a_i(t, \xi) \geq  \gamma^{i} e^{-4r^{\beta} (i-1)t} t^{i-1} e^{-t|\xi|^{\beta}} \chi_{B_0(r)} (\xi)
\end{eqnarray*}
for $i<2n,$ then we have 
\begin{eqnarray*}
a_{2n} (t, \xi) & = &  \int_0^t e^{-(t-s)|\xi|^{\beta}} (2  \int_{\R^d}a_1 (s, \xi -y)  a_{2n-1}(s, y) dy +\cdots \\
&& + 2 \int_{\R^d}a_{n-1} (s, \xi-y)  a_{n+1}(s, y) dy + \int_{\R^d} a_n(s, \xi-y) a_n(s, y) dy) ds\\
& \geq & (2n-1) \gamma^{2n} \int_0^t e^{-(t-s) |\xi|^{\beta}}\\
&& \int_{\R^d} s^{2n-2} e^{-4r^{\beta} (2n-2) s} e^{-s |\xi-y|^{\beta}} e^{-s|y|^{\beta}} \chi_{B_0(r)}(\xi-y) \chi_{B_0(r)} (y) dy ds\\
& \geq &  (2n-1) \gamma^{2n} e^{-4r^{\beta} (2n-2) t} \int_0^t s^{2n-2} e^{-(t-s) |\xi|^{\beta}}  \\
&& \int_{\R^d}  e^{-s |\xi-y|^{\beta}} e^{-s|y|^{\beta}} \chi_{B_0(r)}(\xi-y) \chi_{B_0(r)} (y) dy ds  
\end{eqnarray*}
Dividing integral into two parts as before, we have 
\begin{eqnarray*}
a_{2n} (t, \xi) & \geq & (2n-1) \gamma^{2n} e^{-4r^{\beta} (2n-2) t} \int_0^t  s^{2n-2} \\
&& (\int_{\R^d \cap E_{\xi, y}}  e^{-(t-s) |\xi|^{\beta}}  e^{-s|y|^{\beta}} \chi_{B_0(r)}(\xi-y) \chi_{B_0(r)} (y) dy \\
&& + e^{-t|\xi|^{\beta}}\int_{\R^d \cap F_{\xi, y}}   \chi_{B_0(r)}(\xi-y) \chi_{B_0(r)} (y) dy) ds\\
& \geq &  \gamma^{2n} e^{-4r^{\beta} (2n-2) t} t^{2n-1}  (\int_{\R^d \cap E_{\xi, y}}  e^{-(t-s) |\xi|^{\beta}}  e^{-s|y|^{\beta}} \chi_{B_0(r)}(\xi-y) \chi_{B_0(r)} (y) dy ds\\
&& + e^{-t|\xi|^{\beta}}\int_{\R^d \cap F_{\xi, y}}   \chi_{B_0(r)}(\xi-y) \chi_{B_0(r)} (y) dy) \\
& \gtrsim &  \gamma^{2n} e^{-4r^{\beta} (2n-2) t} t^{2n-1}  (  e^{-4r^2t}e^{-t|\xi|^{\beta}} \int_{\R^d \cap E_{\xi, y}}  \chi_{B_0(r)}(\xi-y) \chi_{B_0(r)} (y) dy\\
&& + e^{-t|\xi|^{\beta}}\int_{\R^d \cap F_{\xi, y}}   \chi_{B_0(r)}(\xi-y) \chi_{B_0(r)} (y) dy) \\
& \gtrsim &  \gamma^{2n} e^{-4r^{\beta} (2n-1) t} t^{2n-1}  e^{-t|\xi|^{\beta}}[\chi_{B_0(r)} \ast \chi_{B_0(r)}](\xi) \\
& \gtrsim &  \gamma^{2n} e^{-4r^{\beta} (2n-1) t} t^{2n-1}  e^{-t|\xi|^{\beta}}\chi_{B_0(r)} (\xi)
\end{eqnarray*}
and 
\begin{eqnarray*}
a_{2n+1} (t, \xi) & \gtrsim &  \gamma^{2n+1} e^{-4r^{\beta} (2n) t} t^{2n}  e^{-t|\xi|^{\beta}}\chi_{B_0(r)} (\xi).
\end{eqnarray*}
Then, by induction we have
\begin{eqnarray*}
a_{i} (t, \xi) & \gtrsim &  \gamma^{i} e^{-4r^{\beta} (i-1) t} t^{i-1}  e^{-t|\xi|^{\beta}}\chi_{B_0(r)} (\xi),  i =1,2,... .
\end{eqnarray*}
Then, by Lemma \ref{rl},  for  $\gamma \geq 4 r^{\beta} e,$ if $T \geq \frac{1}{4r^{\beta}},$ we have 
\begin{eqnarray*}
\left\| \sum_{i=1}^{\infty} u_i \right\|_{L_{T}^{\infty}(X )} 
& \geq & \left\| \sum_{i=1}^{\infty} \left\| u_i \right\|_{\mathcal{F}L^1} \right\|_{L_{T}^{\infty}}\\
& \geq  &  \left\| \sum_{i=1}^{\infty} \left\| \gamma^{i} e^{-4r^{\beta} (i-1) t} t^{i-1} e^{-t|\xi|^2} \chi_{B_0(r)} (\xi) \right\|_{L^1} \right\|_{L_{T}^{\infty}}\\
& \geq  &  \left\| \sum_{i=1}^{\infty}  \gamma^{i} e^{(-4r^{\beta} (i-1)-1) t} t^{i-1}   \left\|\chi_{B_0(r)} (\xi) \right\|_{L^1} \right\|_{L_{T}^{\infty}} = \infty.
\end{eqnarray*}
By  similar argument as above, we can obtain that for general $k,$
\begin{eqnarray*}
a_{ik- (i-1)} (t, \xi) & \gtrsim &  \gamma^{ik - (i-1)} e^{-4r^{\beta}(k-1)i t} t^{i}  e^{-t|\xi|^{\beta}}\chi_{B_0(r)} (\xi),  i =1,2,... .
\end{eqnarray*}
Then, by Lemma \ref{rl},  for  $\gamma \geq 4 r^{\beta}(k-1)e,$ if $T \geq \frac{1}{4r^{\beta}(k-1)},$ we have 
\begin{eqnarray*}
\left\| \sum_{i=1}^{\infty} u_i \right\|_{L^{\infty}_{T}(X )}
& \geq  &  \left\| \sum_{i=1}^{\infty} \left\| \gamma^{i} e^{-4r^{\beta} (i-1) t} t^{i-1}  e^{-t|\xi|^{\beta}}\chi_{B_0(r)} (\xi) \right\|_{L^1} \right\|_{L_{T}^{\infty}}\\
& \geq  &  \left\| \sum_{i=1}^{\infty} \left\| \gamma^{i} e^{(-4r^{\beta} (i-1)-1) t} t^{i-1}  \chi_{B_0(r)} (\xi) \right\|_{L^1} \right\|_{L_{T}^{\infty}}
 =  \infty.
\end{eqnarray*}
\end{proof}
We are now ready to prove Theorem \ref{dip}.
\begin{proof}[Proof of Theorem \ref{dip}]
By Proposition \ref{lw} and Lemmas \ref{ia}, \ref{me1}, we have, for $0<T<T^*,$
\[ u(t)=\sum_{i=1}^{\infty} a_i(t) \in L^{\infty}_{loc} ([0, T^*), X(\m))\] 
is a solution of \eqref{fh}. By Lemma \ref{cl}, 
$ \left\| \sum_{i=1}^{\infty} u_i(t)\right\|_{L^{\infty}_{loc}([0, T^*), X)} \leq \left\| \sum_{i=1}^{\infty} a_i(t)\right\|_{L^{\infty}_{loc}([0, T^*),  X)}.$
By Lemma \ref{fl}, we have, for $T^{*}\geq  \frac{1}{4r^{\beta} (k-1)},$$\left\| \sum_{i=1}^{\infty} u_i(t)\right\|_{L^{\infty}([0, T^*),  X)}=\infty.$
This completes the proof.
\end{proof}
\noindent
{\textbf{Acknowledgment}:}  D.G. B is thankful to DST-INSPIRE (DST/INSPIRE/04/2016/001507) for the research grant.\\
\noindent
\textbf{Declaration}.  There is no conflict of interest for this article.

\bibliographystyle{amsplain}
\bibliography{heat}
\end{document}